\documentclass[12pt,twoside]{article}
\usepackage{etex}
\usepackage[dvipsnames]{xcolor} 
\usepackage{color}
\usepackage{pdfpages}

\usepackage{amsthm}
\usepackage{amsfonts,amsmath,amssymb,amsxtra,url,float} 
\allowdisplaybreaks[4] 
\usepackage[colorlinks,linkcolor=RoyalBlue,anchorcolor=Periwinkle,citecolor=Orange,urlcolor=Green]{hyperref} 
\usepackage{enumitem}
\usepackage{geometry} 
\usepackage{rotating} 
\usepackage{lscape} 
\usepackage{multirow}
\usepackage{graphicx} 
\usepackage{subfigure} 
\usepackage{tikz}
\usetikzlibrary{decorations.pathreplacing,decorations.markings}
 \tikzset{
  on each segment/.style={
    decorate,
    decoration={
      show path construction,
      moveto code={},
      lineto code={
        \path [#1]
        (\tikzinputsegmentfirst) -- (\tikzinputsegmentlast);
      },
      curveto code={
        \path [#1] (\tikzinputsegmentfirst)
        .. controls
        (\tikzinputsegmentsupporta) and (\tikzinputsegmentsupportb)
        ..
        (\tikzinputsegmentlast);
      },
      closepath code={
        \path [#1]
        (\tikzinputsegmentfirst) -- (\tikzinputsegmentlast);
      },
    },
  },
  mid arrow/.style={postaction={decorate,decoration={
        markings,
        mark=at position 0.6 with {\arrow[#1]{stealth}} 
      }}},
}
\usetikzlibrary{arrows}
\usetikzlibrary{trees}

\usetikzlibrary{matrix}
\usetikzlibrary{patterns}
\usetikzlibrary{shadings} 
\usepackage[all]{xy} 
\usepackage{mathdots} 
\usepackage{dsfont} 
\usepackage{cite}
\usepackage{mathrsfs} 
\usepackage{multicol} 
\numberwithin{figure}{section}
\usepackage{marginnote} 
\usepackage{graphicx} 
\usepackage{multicol} 

\usepackage{enumitem}
\setenumerate[1]{itemsep=0pt,partopsep=0pt,parsep=\parskip,topsep=3pt}
\setitemize[1]{itemsep=0pt,partopsep=0pt,parsep=\parskip,topsep=3pt}
\setdescription{itemsep=0pt,partopsep=0pt,parsep=\parskip,topsep=3pt}
\setlist[itemize]{leftmargin=35pt}

\newtheorem{theorem}{Theorem}[section]
\newtheorem{lemma}[theorem]{Lemma}
\newtheorem{corollary}[theorem]{Corollary}
\newtheorem{main theorem}[theorem]{Main Theorem}

\newtheorem{definition}[theorem]{Definition}

\newtheorem{remark}[theorem]{Remark}
\newtheorem{example}[theorem]{Example}

\newtheorem{assumption}[theorem]{Assumption}

\usetikzlibrary{arrows}

\numberwithin{equation}{section}

\def\<{\langle} 
\def\>{\rangle} 
\def\NN{\mathbb{N}} 

\newcommand{\Pic}{Figure\ }
\newcommand{\modcat}{\mathsf{mod}}

\newcommand{\proj}{\mathsf{proj}}
\newcommand{\Gproj}{\mathsf{G}\text{-}\mathsf{proj}}
\newcommand{\npGproj}{\mathsf{G}\text{-}\mathsf{proj}_{\pmb{\oslash}}}

\newcommand{\ind}{\mathsf{ind}}
\newcommand{\kk}{\mathds{k}} 
\newcommand{\Q}{\mathcal{Q}} 
\newcommand{\I}{\mathcal{I}} 
\newcommand{\e}{\varepsilon}
\newcommand{\Hom}{\mathrm{Hom}} %
\newcommand{\End}{\mathrm{End}} %
\newcommand{\Ext}{\mathrm{Ext}} %

\newcommand{\ol}[1]{\overline{#1}}
\newcommand{\w}[1]{\tilde{#1}}
\newcommand{\To}[1]{\mathop{-\!\!\!-\!\!\!\longrightarrow}\limits^{#1}}

\def\alg{\mathit{\Lambda}}
\def\C{\mathscr{C}}
\def\s{\mathfrak{s}}
\def\t{\mathfrak{t}}
\def\emb{\mathfrak{e}}
\def\calR{\mathcal{R}}
\def\GcalR{\mathrm{G}\mathcal{R}}
\def\res{\mathrm{res}}

\newcommand{\defines}{\it\color{black}}
\title{ \bf Recollements and Gorenstein projective modules for gentle algebras
\
}
\vspace{0.2cm}

\author{%
Yu-Zhe Liu$^1$, Dajun Liu$^2$, Xin Ma$^{3,\ddag}$\\
{\footnotesize 1. School of Mathematics and Statistics, Guizhou University, Guiyang 550025, Guizhou, P. R. China;}\\
{\footnotesize E-mail: liuyz@gzu.edu.cn / yzliu3@163.com;
  ORCID: \href{https://orcid.org/0009-0005-1110-386X}{0009-0005-1110-386X}.} \\
{\footnotesize 2. Anhui Polytechnic University, Wuhu 241000 Anhui, P. R. China;}\\
{\footnotesize E-mail: liudajun@ahpu.edu.cn;
  ORCID: \href{https://orcid.org/0009-0001-6073-7587}{0009-0001-6073-7587}. }
\\
{\footnotesize 3. College of Science, Henan University of Engineering, Zhengzhou 451191, Henan, P. R. China;}
\\
{\footnotesize E-mail: maxin@haue.edu.cn.
  }
\\
{\footnotesize $\ddag$: Corresponding Author}
}
\date{ }
\begin{document}






\maketitle

\begin{abstract}

Let $A=\kk\Q/\I$ be a gentle algebra.
We provide a bijection between non-projective indecomposable Gorenstein
projective modules over $A$ and special recollements induced by an arrow $a$ on
any full-relational oriented cycle $\C$, which satisfies some interesting
properties, for example, the tensor functor $-\otimes_A A/A\e A$ sends
Gorenstein projective module $aA$ to an indecomposable projective $A/A\e A$-module; and $-\otimes_A A/A\e A$ preserves Gorenstein projective objects if any
two full-relational oriented cycles do not have common vertex.


\end{abstract}



\section{Introduction}
Recollements of Abelian and triangulated categories were introduced by Be{\u\i}linson, Bernstein and Deligne \cite{BBD1982} in
connection with derived categories of sheaves on topological spaces with the idea that one triangulated category may be ``glued together'' from two others, which play an important role in
representation theory of algebras.
Reduction techniques with respect to a recollement of triangulated or Abelian categories have been investigated widely, see for example,
\cite[etc]{ChX17R,H93R,H14R, Lu2017,MZH19G,Psa2014, PV19P,Q16R,Zh23R, ZZ2020,ZZ2020JAA}.
We use $\modcat(A)$ to denote the category of finitely generated right modules  over an Artinian algebra $A$,
and use $\Gproj(A)$ to denote the full subcategory of $\modcat(A)$ containing all Gorenstein-projective (=G-projective for shrot) modules.
In particular, in \cite{Lu2017}, for two Artinian algebras $A$ and $B$ and a functor $F_1: \modcat(A) \to \modcat(B)$ preserving projective objects,
Lu showed that if $F_1$ has a right adjoint functor $F_2$ and satisfies one of the following conditions:
\begin{itemize}
  \item[$\spadesuit$] $F_1$ is a functor such that the following two conditions hold:
    \begin{itemize}
      \item $\Ext_A^k(X, F_2(Q))=0$  for any projective right $B$-module $Q$, any G-projective module $X$, and any $k>0$;
      \item for arbitrary G-projective right $A$-module $X_1$,
      the short exact sequence $0\to X_2 \to P \to X_1 \to 0$ in $\modcat(A)$ ($P$ is projective),
      we obtain that 
      $0\to F_1(X_2) \to F_1(P) \to F_1(X_1) \to 0$ is also exact;
    \end{itemize}
  \item[$\clubsuit$] $F_1|_{\Gproj(A)}$ is exact, and, for any projective right $B$-module $Q$,
  $\mathrm{proj.dim}F_2(Q)$, the projective dimension of $F_2(Q)$, is finite
  or $\mathrm{inj.dim}F_2(Q)$, the injective dimension of $F_2(Q)$, is finite;
\end{itemize}
then $F_1$ preserves G-projectives.

In this paper, we will provide a functor $F_1$ does not satisfy $\spadesuit$ and $\clubsuit$ by using some special recollement for gentle algebra such that $F_1$ preserves G-projective, see Corollary \ref{coro:of main 1}.
To do this, we provide a method to describe G-projective modules by some special recollements, see Theorem \ref{thm:main 1}.

We focus on some special recollements for gentle algebras and
provide a reduction on non-projective indecomposable G-projective modules in the middle category of the recollements.
This work began with the works of Kalck \cite{Kal2015}, Chen-Lu \cite{CL2017, CL2019}, Li-Zhang \cite{LiZ2021}, and our research on the G-projective modules and $\tau$-tilting theory over gentle algebras in \cite{LZHpre}.
Kalck originally provided the descriptions of G-projective modules and characterized the singularity categories for gentle algebras, see \cite[Theorem 2.5]{Kal2015}.
In \cite{LZHpre}, the authors considered the G-projective support $\tau$-tilting modules, introduced in \cite{LiZ2021}, over gentle algebras,
and showed that a gentle algebra is representation-finite if and only if, for any G-projective support $\tau$-tilting module $G$, the endomorphism algebra $\End_AG$ is representation-finite by using marked surfaces introduced by Baur$-$Coelho-Sim{\~o}es \cite{BCS2021}.
This conclusion provides a description of the representation-type of gentle algebras by using G-projective modules.
Our main results provide a description of non-projective indecomposable G-projective modules over gentle algebras by using recollements.

Assume that $\kk$ is an algebraically closed field.
Let $A=\kk\Q/\I$ be a gentle algebra. Assume that the bound quiver $(\Q,\I)$ contains at least one full-relational oriented cycle (see Subsection \ref{subsect:gent}) $\C=a_1\cdots a_{\ell(\C)}$ of length $\ell(\C)$,
and the lengths of all full-relational oriented cycles are greater than or equal to $3$.
Then each arrow $a_t$ corresponds to a non-projective indecomposable G-projective module $a_tA$ by {\rm\cite[Theorem 2.5]{Kal2015}} {\rm(}see Theorem \ref{thm:Kalck}{\rm)},
and corresponds to a recollement
\hspace{-0.5cm}\[\calR_{\C,t} := \ \ \
 \xymatrix@C=2cm{
  \modcat(\ol{A}_{\C,t})
  \ar[r]^{\emb_{\C,t}}
& \modcat(A)
  \ar[r]^{\res_{\C,t}}
  \ar@/^1.8pc/[l]^{H_{\C,t}}
  \ar@/_1.8pc/[l]_{T_{\C,t}}
& \modcat(\w{A}_{\C,t})
  \ar@/_1.8pc/[l]_{\w{T}_{\C,t}}
  \ar@/^1.8pc/[l]^{\w{H}_{\C,t}}
}\]
by the idempotent \[\epsilon_{\C,t} = \sum_{i\ne \s(a_t), \t(a_t)} \e_i, \]
see Section \ref{sect:main},
where $e_v$ is the idempotent of $A$ corresponded by the vertex $v$ of $\Q$,
$\ol{A}_{\C,t}:=A/A\epsilon_{\C,t}A$, and $\w{A}_{\C,t}:=\epsilon_{\C,t}A\epsilon_{\C,t}$.
The following theorem is the first main result of our paper.

\begin{theorem}[Theorem \ref{thm:main 1}]
The map
\[\varphi: \ind(\npGproj(A)) \to \GcalR(A), \
a_tA \mapsto \calR_{\C, t}\]
from the set $\ind(\npGproj(A))$ of all non-projective indecomposable G-projective right $A$-modules {\rm(}up to isomorphism{\rm)} to the set $\GcalR(A):=\{\calR_{\C,t} \mid \C \text{ is a full-relational cycle}, 1\le t\le \ell(\C)\}$ is a bijection such that the following statements hold.
\begin{itemize}
  \item
    The right $\ol{A}_{\C,t}$-module $T_{\C,t}(G)$ is indecomposable and projective for any right $A$-module $G\in\ind(\npGproj(A))$.
  \item
    If there is a right $A$-module $G\in\ind(\npGproj(A))$ such that the dimension of $\res_{\C,t}(G)$, as a $\kk$-linear space, is greater than or equals to $2$, then $A$ is representation-infinite.
\end{itemize}
\end{theorem}

The above theorem has a non-trivial corollary as follows.

\begin{corollary}[Corollary \ref{coro:of main 1}]
If arbitrary two full-relational oriented cycles of a gentle algebra $A=\kk\Q/\I$ have no common vertex, then for any full-relational oriented cycle $\C=a_1\cdots a_{ \ell(\C)}$ and arbitrary $1\le t\le \ell(\C)$, the functor $T_{\C,t}: \modcat(A) \to \modcat(\ol{A}_{\C,t})$ of the recollement $\calR_{\C,t}$ preserves G-projectives.
\end{corollary}

In the case of $A$ to be gentle one-cycle, then its bound quiver contains only one cycle $\C$. We still assume that $\C=a_1\cdots a_{\ell(\C)}$ is full-relational oriented,
then the functors $T_{\C,t}$ and $\res_{\C,t}$ of the recollement $\varphi(a_tA) = \calR_{\C,t}$ send $a_uA \in\ind(\npGproj(A))$ to an indecomposable projective module $T_{\C,t}(a_uA)\in \modcat(\ol{A}_{\C,t})$
and a restriction $\res_{\C,t}(a_uA) \in \modcat(\w{A}_{\C,t})$, respectively.
Furthermore, the embedding $\emb_{\C,t}: \modcat(\ol{A}_{\C,t}) \to \modcat(A)$ sends every $T_{\C,t}(a_uA)$ to a right $A$-module, isomorphic to zero or $a_tA$, lying in $\ind(\npGproj(A))$.
See the second main result as follows.

\begin{theorem}[Theorem \ref{thm:main 2}]
If $\C=a_1\cdots a_{\ell(\C)}$ is a unique cycle of a gentle algebra $A$, and $\C=a_1\cdots a_{\ell(\C)}$ is full-relational oriented, then the following statements hold.
\begin{itemize}
  \item[\rm(1)] $T_{\C,t}(a_uA) \cong
    \begin{cases}
      \ol{e}_{\t(a_t)}\ol{A}_{\C,t}, & \text{if } u=t; \\
      0, & \text{if } u\ne t
    \end{cases}$ holds for all $1\le u\le \ell(\C)$;
  \item[\rm(2)] for any indecomposable projective right $\ol{A}_{\C,t}$-module $P$
    satisfying $P \not\cong \ol{e}_{\t(a_t)}\ol{A}_{\C,t}$,
    we have $\emb_{\C,t}(P) \notin \ind(\npGproj(A))$;
  \item[\rm(3)] $\emb_{\C,t}(T_{\C,t}(a_uA))$ is a G-projective right $A$-module.
\end{itemize}
\end{theorem}

\section{Preliminaries} \label{sect:pre}

In this section, we will give some terminologies and some preliminary results.

\subsection{Gentle algebras} \label{subsect:gent}

Let $\Q$ be a quiver and $\underline{\Q}$ its underlying graph. A {\defines cycle} (of length $l$) on $\Q$ is a cycle on $\underline{\Q}$, that is, it is a sequence of $l$ edges $\underline{c}_1, \ldots, \underline{c}_l$ of $\underline{\Q}$ with $n$ vertices $v_1,\ldots, v_n \in \Q_0$ such that the vertices of $\C$ can be arranged in a cyclic sequence in such a way that two vertices $v_i$ and $v_{i+1}$ are adjacent connected by the arrow $c_i$ if they are consecutive in the sequence, and are nonadjacent otherwise (the indices $i$ are taken modulo $n$ if necessary).
An {\defines oriented cycle} is a cycle $\C=a_1\cdots a_l$ with $\t(a_{i})=\s(a_{i+1})$  $(1\le i<l)$ such that $\t(\C)=\t(a_l)=\s(a_1)=\s(\C)$ holds.
Furthermore, an oriented cycle of a bound quiver $(\Q,\I)$ is said to be {\defines full-relational} if $a_1a_2$, $a_2a_3$, $\ldots$, $a_{l-1}a_l$ and $a_la_1$ lie in $\I$.

Next, we recall that a bound quiver $(\Q, \I)$ is said to be a {\defines gentle pair} if the following conditions hold:
\begin{itemize}
  \item[(1)] Each vertex in $\Q_0$ is the source of at most two arrows and the target of at most two arrows.

  \item[(2)] For each arrow $a\in\Q_1$, there is at most one arrow $b\in\Q_1$ such that $ab\notin\I$, and there is at most one arrow $c$ such that $ca\notin\I$.

  \item[(3)] For each arrow $a\in\Q_1$, there is at most one arrow $b\in\Q_1$ such that $ab\in\I$, and there is at most one arrow $c$ such that $ca\in\I$.

  \item[(4)] The admissible $\I$ of the path algebra $\kk\Q$ is generated by some paths of length two.
\end{itemize}

\begin{definition} \rm
An algebra $\kk\Q/\I$ is called a {\defines gentle algebra} if its bound quiver $(\Q,\I)$ is a gentle pair. Furthermore, a gentle algebra is said to be a {\defines gentle one-cycle algebra}, if its quiver $\Q$ has only one cycle.
\end{definition}

Gentle algebras were introduced by Assem and Skowr\`{o}nski in \cite{AS1987} as appropriate context for the study of
algebras derived equivalently to hereditary algebras of Euclidean type $\w{\mathbb{A}}$,
they are special string algebras, and all indecomposable modules over gentle algebras are described by Butler and Ringel, see \cite[Section 3, page 161]{BR1987}.

\begin{example} \rm \label{exp:gent}
The quiver $\Q$ given by the following graph (see \Pic \ref{fig:gent}) with the admissible ideal $\I=\langle a_1a_2, a_2a_3, a_3a_1, b_1c_1, b_2c_2, b_3c_3, d_1b_3, d_2b_1, d_3b_2\rangle$ is a gentle pair.
It is clear that $A=\kk\Q/\I$ is a gentle algebra whose cycle is a {full-relational} oriented cycle $a_1a_2a_3\ (= a_2a_3a_1 = a_3a_1a_2)$.
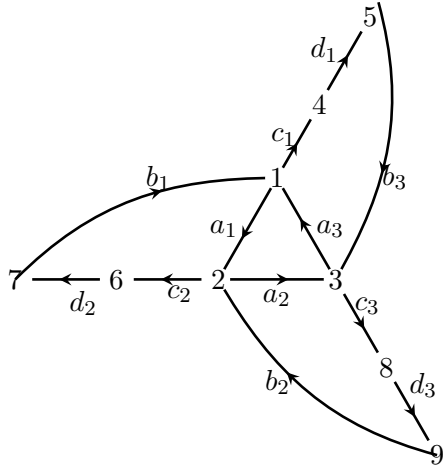
\begin{figure}[htbp]
\centering
\small
\begin{tikzpicture}[scale = 0.45]
\draw[rotate around={60:(0,2)}]
  [line width=1pt]
  [postaction={on each segment={mid arrow=black}}]
  (0,2) -- (2,2);
\draw[rotate around={60:(0,2)}]
  (2.5,2) node{$4$};
\draw[rotate around={60:(0,2)}]
  [line width=1pt]
  [postaction={on each segment={mid arrow=black}}]
  (3,2) -- (5,2);
\draw[rotate around={60:(0,2)}]
  (5.5,2) node{$5$};
\draw[line width=1pt]
  [postaction={on each segment={mid arrow=black}}]
  (-2.23,-1)--(-4.23,-1);
\draw (-4.23,-1) node[left]{$6$};
\draw[line width=1pt]
  [postaction={on each segment={mid arrow=black}}]
  (-5.23,-1)--(-7.23,-1);
\draw (-7.23,-1) node[left]{$7$};
\draw[rotate=120]
  [line width=1pt]
  [postaction={on each segment={mid arrow=black}}]
  (-2.23,-1)--(-4.23,-1);
\draw[rotate=120] (-4.73,-1) node{$8$};
\draw[rotate=120]
  [line width=1pt]
  [postaction={on each segment={mid arrow=black}}]
  (-5.23,-1)--(-7.23,-1);
\draw[rotate=120] (-7.73,-1) node{$9$};
\draw[rotate =  0][line width=1pt]
  [postaction={on each segment={mid arrow=black}}]
   (-7.73,-1) to[out=45,in=180] (0, 2);
\draw[rotate =120][line width=1pt]
  [postaction={on each segment={mid arrow=black}}]
   (-7.73,-1) to[out=45,in=180] (0, 2);
\draw[rotate =240][line width=1pt]
  [postaction={on each segment={mid arrow=black}}]
   (-7.73,-1) to[out=45,in=180] (0, 2);
\draw[rotate =  0](-3.5,2  ) node{$b_1$};
\draw[rotate =120](-3.5,2  ) node{$b_2$};
\draw[rotate =240](-3.5,2  ) node{$b_3$};
\draw[rotate =  0]( 0.2,3.2) node{$c_1$};
\draw[rotate =120]( 0.2,3.2) node{$c_2$};
\draw[rotate =240]( 0.2,3.2) node{$c_3$};
\draw[rotate =  0]( 1.4,5.8) node{$d_1$};
\draw[rotate =120]( 1.4,5.8) node{$d_2$};
\draw[rotate =240]( 1.4,5.8) node{$d_3$};
\draw[postaction={on each segment={mid arrow=black}}]
     [line width=1pt]
     ( 0.  , 2.  ) -- (-1.73,-1.  );
\draw(-0.86, 0.5 ) node[left]{$a_1$};
\draw[rotate=120][postaction={on each segment={mid arrow=black}}]
     [line width=1pt]
     ( 0.  , 2.  ) -- (-1.73,-1.  );
\draw[rotate=120](-0.86, 0.5 ) node[below]{$a_2$};
\draw[rotate=240][postaction={on each segment={mid arrow=black}}]
     [line width=1pt]
     ( 0.  , 2.  ) -- (-1.73,-1.  );
\draw[rotate=240](-0.86, 0.5 ) node[right]{$a_3$};
\fill[rotate=  0][white]( 0.  , 2.  ) circle (0.35);
\fill[rotate=120][white]( 0.  , 2.  ) circle (0.35);
\fill[rotate=240][white]( 0.  , 2.  ) circle (0.35);
\draw[rotate=  0]( 0.  , 2.  ) node{$1$};
\draw[rotate=120]( 0.  , 2.  ) node{$2$};
\draw[rotate=240]( 0.  , 2.  ) node{$3$};
\end{tikzpicture}
\caption{The quiver of gentle algebra given in Example \ref{exp:gent}.}
\label{fig:gent}
\end{figure}
\end{example}

\subsection{G-projective modules}

Let $A$ be a finite-dimensional algebra.
A right $A$-module $G\in\modcat(A)$ is called {\it Gorenstein projective} ({\it G-projective}, for short)
if there is an exact sequence of projective right $A$-modules
\[\cdots \longrightarrow P^{-2}
\mathop{\longrightarrow}\limits^{d^{-2}} P^{-1} \mathop{\longrightarrow}\limits^{d^{-1}} P^{0}
\mathop{\longrightarrow}\limits^{d^{0}}  P^{1}  \mathop{\longrightarrow}\limits^{d^{1}}
P^{2} \longrightarrow \cdots \]
in $\modcat A$ which remains exact after applying the functor $\Hom_A(-,A)$,
such that $G\cong \mathrm{Im}d^{-1}$ \cite{AB1969, EJ1995}.
Obviously, every projective right $A$-module is G-projective.
We use $\Gproj(A)$ to denote the subcategory of $\modcat(A)$ consisting of all G-projective right $A$-modules,
and $\npGproj(A)$ to denote the subcategory of $\modcat(A)$ consisting of all non-projective G-projective right $A$-modules.
In \cite{Kal2015}, Kalck described the G-projective modules over gentle algebras.

\begin{theorem}[\!\!{\cite[Theorem 2.5]{Kal2015}}] \label{thm:Kalck}
Let $A=\kk\Q/\I$ be a gentle algebra. Then any indecomposable right $A$-module $G$ is G-projective if and only if $G$ is isomorphic to either $eA$
{\rm(}$e$ is a primitive idempotent of $A${\rm)} or $aA$ {\rm(}$a$ is an arbitrary arrow on any full-relational oriented cycle{\rm)}.
In particular, all $aA$ are indecomposable non-projective.
\end{theorem}

\begin{remark} \rm
(1) In \cite{CSZ2018}, Chen-Shen-Zhou extended Kalck's results to monomial algebras and showed that the G-projective modules over a monomial algebra $\alg$ is of the form $p\alg$, where $p$, say a perfect path, is a special path on some special oriented cycle on the bound quiver of $\alg$.

(2) For a gentle algebra $A$, it is clear that $\ind(\npGproj(A)) \ne \varnothing$ if and only if its bound quiver has at least one full-relational oriented cycle, see, for example, \cite{Kal2015, LGH2024}.
\end{remark}

\begin{example} \label{exp:gent G-proj} \rm
Consider the gentle algebra given in Example \ref{exp:gent}. There are three indecomposable and non-projective G-projective right $A$-modules
\[
a_1A \cong \left(
\begin{smallmatrix}
2\\
6\\
7
\end{smallmatrix}
\right),
a_2A \cong \left(
\begin{smallmatrix}
3\\
8\\
9
\end{smallmatrix}
\right),
\text{ and }
a_3A \cong \left(
\begin{smallmatrix}
1\\
4\\
5
\end{smallmatrix}
\right). \
\]
\end{example}

\subsection{ Recollements}
We recall the notion of recollements of Abelian categories.
\begin{definition}[\!\!{\cite{FP2004}}]
\label{def-2.1} \rm
A recollement, denoted by $\calR(\mathcal{A},\mathcal{B},\mathcal{C})$, of Abelian categories is a diagram
\begin{align} \label{recollement diagram}
\xymatrix@C=2cm{
  \mathcal{A}
  \ar[r]^{i}
& \mathcal{B}
  \ar[r]^{e}
  \ar@/_1.5pc/[l]_{q}
  \ar@/^1.5pc/[l]^{p}
& \mathcal{C},
  \ar@/_1.5pc/[l]_{l}
  \ar@/^1.5pc/[l]^{r}
}
\end{align}
of Abelian categories and additive functors such that
\begin{itemize}
  \item[(1)] $(q,i)$, $(i,p)$, $(l,e)$, $(e,r)$ are adjoint pairs;
  \item[(2)] the functors $i$, $l$, and $r$ are fully faithful;
  \item[(3)] $\mathrm{Im}(i)=\mathrm{Ker}(e)$.
\end{itemize}
\end{definition}

The following example is widely studied, which plays a crucial role in the sequel.

\begin{example}\label{thm:recollement} \rm
(\!\!\cite[Example 2.7]{Psa2014})
Let $\e$ be an idempotent of an algebra $A$. Then we have a recollement of module categories:
\begin{align}\label{recollement}
\xymatrix@C=2cm{
  \modcat(A/A\e A)
  \ar[r]^{\emb}_{\mathrm{embedding}}
& \modcat(A)
  \ar[r]^{(-)\e}_{\text{retraction}}
  \ar@/_1.5pc/[l]_{-\otimes_A A/A\e A}
  \ar@/^1.5pc/[l]^{\Hom_A(A/A\e A,-)}
& \modcat(\e A \e),
  \ar@/_1.5pc/[l]_{-\otimes_{\e A\e} \e A}
  \ar@/^1.5pc/[l]^{\Hom_{\e A\e}(A\e,-)}
}
\end{align}
where $\emb$ is an embedding functor.

\end{example}

%

\section{ The idempotents on oriented cycles}

In this section, we assume that the following assumption holds.

\begin{assumption} \label{assump} \rm
All finite-dimensional algebras we considered in this section are gentle algebras whose all full-relational oriented cycles are cycles of length $\ge 3$.
\end{assumption}

For a gentle pair $(\Q,\I)$ with a full-relational oriented cycle $\C=a_1a_2\cdots a_l$ ($\s(a_i)=i$, $1\le i\le l$, $\t(a_l)=1=\s(a_1)$), we define
\[\epsilon_{\C,t} = \sum_{i\ne \s(a_t),\t(a_t)} \e_i \text{ and } \ol{A}_{\C,t}=A/A\epsilon_{\C,t}A, \]
where, for any vertex $v\in\Q_0$, $\e_v$ is the idempotent corresponded by $v$.
Notice that $\ol{A}_{\C,t}$ is both a left $A$-module and right $\ol{A}_{\C,t}$-module (the left $A$-action $A\times \ol{A}_{\C,t} \to \ol{A}_{\C,t}$ is defined by
$(a, x+A\epsilon_{\C,t}A) \mapsto ax+A\epsilon_{\C,t}A$.).


\subsection{The quotient {\texorpdfstring{$\ol{A}_{\C,t}$} .}} \label{subsect:quotient}

For any element $x$ in $A$, we use $\ol{x}$ to represent the image of $x$ under the canonical epimorphism $A \to \ol{A}_{\C,t} = A/A\epsilon_{\C,t}A$,
and, without causing confusion, $p\ol{A}_{\C,t}$ is the right $\ol{A}_{\C,t}$-module $\ol{p}\ol{A}_{\C,t}$ for any path $p$ of length $\ge 1$ on the quiver of $A$.
The following lemma shows that $a_tA\otimes_A \ol{A}_{\C,t}$ is an indecomposable projective right $\ol{A}_{\C,t}$-module.

\begin{lemma} \label{lemm:idemp}
Let $A=\kk\Q/\I$ be a gentle algebra.
Then, for a full-relational oriented cycle $\C=a_1a_2\cdots a_l$ of $A$,
the tensor product $a_tA\otimes_A \ol{A}_{\C,t}$ is an indecomposable projective right $\ol{A}_{\C,t}$-module.
\end{lemma}


\begin{proof}
Notice that $a_tA\otimes_A \ol{A}_{\C,t} \cong a_t\ol{A}_{\C,t}$.
Next, we show that $a_t\ol{A}_{\C,t}$ is isomorphic to the indecomposable projective right $\ol{A}_{\C,t}$-module corresponded by the vertex $\t(a_t)$ of the quiver of $\ol{A}_{\C,t}$,
that is, we show $a_t\ol{A}_{\C,t} \cong \ol{\e}_{\t(a_t)}\ol{A}_{\C,t}$ in this proof.

First of all, we have
\begin{align}\label{formula:sum 1}
  a_t\ol{A}_{\C,t}
= a_t
  \bigoplus_{
    \begin{smallmatrix}
      {p\in\Q_s} \\ {s\in\NN}
    \end{smallmatrix}
    } \kk p
\mathop{=\!=}\limits^{(\star)}
  \bigoplus_{
    \begin{smallmatrix}
      p \text{ does not cross } \\ {1,\ldots,t, t+2, \ldots, l}
    \end{smallmatrix}
    } \kk a_t p
\end{align}
and
\begin{align}\label{formula:sum 2}
  \ol{\e}_{\t(a_t)}\ol{A}_{\C,t}
= \ol{\e}_{\t(a_t)}
  \bigoplus_{
    \begin{smallmatrix}
      {q\in\Q_s} \\ {s\in\NN}
    \end{smallmatrix}
    } \kk q
= \bigoplus_{\s(q)=\t(a_t)} \kk q,
\end{align}
where ``$\oplus$'' is a direct sum of $\kk$-linear spaces, and $(\star)$ holds by the following reasons:
\begin{itemize}
  \item[(a)] It is trivial that $p$ does not cross $1,\ldots,t-1,t+2,\ldots,l$
    by $\ol{A}_{\C,t}=A/A\epsilon_{\C,t} A$ and the definition of $\epsilon_{\C,t}$.
  \item[(b)] If $t$ is a vertex on $p$, then $p$ has a subpath $\wp=\alpha_1\cdots\alpha_n$ ($\alpha_1, \cdots, \alpha_n \in \Q_1$) such that $a_t\wp=a_t\alpha_1\cdots\alpha_n$ is an oriented cycle.
    Thus, $a_{t-1}$ and $\alpha_n$ are two arrows with $\t(a_{t-1})=\t(\alpha_n)=\s(a_t)$.
    By the definition of gentle pair and underlying Assumption \ref{assump}, $a_{t-1}a_t=0$ yields that $\alpha_na_t \ne 0$.
    Then we obtain that $\alpha_t\wp$ is an oriented cycle without relation.
    This is a contradiction since $A$ is a finite-dimensional $\kk$-algebra.
\end{itemize}

Now, we show that the set $X_1$ of all direct summands of (\ref{formula:sum 1}) one-to-one corresponds to the set $X_2$ of all direct summands of (\ref{formula:sum 2}).
On the one hand, for arbitrary $\kk a_tp \in X_1$, the starting point of $p$ is $\t(a_t)$.
It follows that $\kk p$ is a direct summand of $\kk p \in X_2$.
Conversely, for any element $\kk q\in X_2$, the path $q$ does not cross $1,\ldots, t, t+2, \ldots, l$ since the images of $\e_{1}, \ldots, \e_{t}, \e_{t+2}, \ldots, \e_l$ under the canonical epimorphism $A \to \ol{A}_{\C,t} = A/A\epsilon_{\C,t} A$ equal zero.
Thus, one can check that there is a bijection between $X_1$ and $X_2$ given by $h: a_tp\mapsto p$.
Furthermore, the above bijection induces a $\kk$-linear isomorphism
\begin{center}
$h: a_t\ol{A}_{\C,t} \To{\cong} \ol{\e}_{\t(a_t)}\ol{A}_{\C,t}$.
\end{center}
On the other hand, the $\kk$-linear isomorphism $h$ is an $A$-homomorphism since $h(\kk a_tp r) = h(\kk a_t(pr)) = \kk (pr) = (\kk p)r = h(\kk a_t p) r$ holds for all paths $r$ on the quiver of $\ol{A}_{\C, t}$.
Therefore, we have $a_t\ol{A}_{\C,t} \cong \ol{\e}_{\t(a_t)}\ol{A}_{\C,t}$ as required.
\end{proof}

Let $A=\kk\Q/\I$ be a gentle algebra.
We use $\ol{\Q}$ and $\ol{\I}$ to denote the quiver and the admissible ideal of $\ol{A}_{\C,t}$ respectively, that is
$\ol{A}_{\C,t} = \kk\ol{\Q}/\ol{\I}$.
Now, we provide an instance for Lemma \ref{lemm:idemp}.

\begin{example} \rm \label{examp:gent-quot}
Consider the gentle algebra $A=\kk\Q/\I$ given in Example \ref{exp:gent}, it has an oriented cycle $\C=a_1a_2a_3$.
Taking $t=1$, then $\epsilon_{\C,1} = \e_3$. So we have $\ol{A}_{\C,1} = A/A\e_3A = \kk\ol{\Q}/\ol{\I}$ whose quiver $\ol{\Q}$ is shown in \Pic \ref{fig:gent quotient}, and the admissible ideal $\ol{\I}$ is $\langle b_1c_1, b_2c_2, d_2b_1, d_3b_2 \rangle$.
\begin{figure}[htbp]
\centering
\small
\begin{tikzpicture}[scale = 0.45]
\draw[rotate around={60:(0,2)}]
  [line width=1pt]
  [postaction={on each segment={mid arrow=black}}]
  (0,2) -- (2,2);
\draw[rotate around={60:(0,2)}]
  (2.5,2) node{$4$};
\draw[rotate around={60:(0,2)}]
  [line width=1pt]
  [postaction={on each segment={mid arrow=black}}]
  (3,2) -- (5,2);
\draw[rotate around={60:(0,2)}]
  (5.5,2) node{$5$};
\draw[line width=1pt]
  [postaction={on each segment={mid arrow=black}}]
  (-2.23,-1)--(-4.23,-1);
\draw (-4.23,-1) node[left]{$6$};
\draw[line width=1pt]
  [postaction={on each segment={mid arrow=black}}]
  (-5.23,-1)--(-7.23,-1);
\draw (-7.23,-1) node[left]{$7$};
%
\draw[rotate=120] (-4.73,-1) node{$8$};
\draw[rotate=120]
  [line width=1pt]
  [postaction={on each segment={mid arrow=black}}]
  (-5.23,-1)--(-7.23,-1);
\draw[rotate=120] (-7.73,-1) node{$9$};
\draw[rotate =  0][line width=1pt]
  [postaction={on each segment={mid arrow=black}}]
   (-7.73,-1) to[out=45,in=180] (0, 2);
\draw[rotate =120][line width=1pt]
  [postaction={on each segment={mid arrow=black}}]
   (-7.73,-1) to[out=45,in=180] (0, 2);
%
\draw[rotate =  0](-3.5,2  ) node{$b_1$};
\draw[rotate =120](-3.5,2  ) node{$b_2$};
\draw[rotate =  0]( 0.2,3.2) node{$c_1$};
\draw[rotate =120]( 0.2,3.2) node{$c_2$};
\draw[rotate =  0]( 1.4,5.8) node{$d_1$};
\draw[rotate =120]( 1.4,5.8) node{$d_2$};
\draw[rotate =240]( 1.4,5.8) node{$d_3$};
\draw[postaction={on each segment={mid arrow=black}}]
     [line width=1pt]
     ( 0.  , 2.  ) -- (-1.73,-1.  );
\draw(-0.86, 0.5 ) node[left]{$a_1$};
\fill[rotate=  0][white]( 0.  , 2.  ) circle (0.35);
\fill[rotate=120][white]( 0.  , 2.  ) circle (0.35);
\fill[rotate=240][white]( 0.  , 2.  ) circle (0.35);
\draw[rotate=  0]( 0.  , 2.  ) node{$1$};
\draw[rotate=120]( 0.  , 2.  ) node{$2$};
\end{tikzpicture}
\caption{The quiver of the quotient $\ol{A}_{\C,1}$.}
\label{fig:gent quotient}
\end{figure}
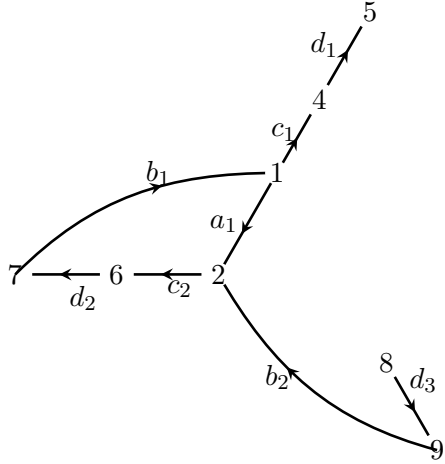
The non-projective indecomposable G-projective right $A$-module
$a_1A \cong \left(
\begin{smallmatrix}
2\\
6\\
7
\end{smallmatrix}
\right)_A$
(see Example \ref{exp:gent G-proj}) corresponds to an indecomposable projective right $\ol{A}_{\C,1}$-module $\ol{\e}_2\ol{A}_{\C,1} $ by the following:
\begin{align*}
  -\otimes_A \ol{A}_{\C,1}:
  \modcat(A)
& \to \modcat(\ol{A}_{\C,1}), \\
  a_1A
& \mapsto a_1A\otimes_A \ol{A}_{\C,1} \\
& \cong a_1\ol{A}_{\C,1} \\
& = \kk a_1 \oplus \kk a_1c_2 \oplus \kk a_1c_2d_2 \\
& \cong
  {
    \left(
    \begin{smallmatrix}
     2\\
     6\\
     7
    \end{smallmatrix}
    \right)_{\ol{A}_{\C,1}}
  } \\
& \cong \ol{\e}_2\ol{A}_{\C,1}.
\end{align*}
\end{example}

\begin{lemma} \label{lemm:proj}
Let $A=\kk\Q/\I$ be a gentle algebra with a full-relational oriented cycle
$\C=a_1a_2\cdots a_l$,
and let $P$ be an indecomposable projective right $\ol{A}_{\C,t}$-module which is not isomorphic to $\ol{\e}_{\t(a_t)}\ol{A}_{\C,t}$.
If the vertex $v\in\ol{\Q}_0$ corresponded by $P=\ol{\e}_v\ol{A}_{\C,t}$, as a vertex of the quiver $\Q$ of $A$, satisfies one of the following conditions:
\begin{itemize}
  \item[{\rm(a)}] $v$ is a vertex which does not to be on any full-relational oriented cycle;
  \item[{\rm(b)}] $v$ is a vertex on the full-relational oriented cycle $\C$ of $\Q$;
  \item[{\rm(c)}] $v$ is a vertex on the other full-relational oriented cycles $\w{\C}$ of $\Q$, and if $\C$ and $\w{\C}$ have at least one common vertex, then any common vertex is either $\s(a_t)$ or $\t(a_t)$;
\end{itemize}
then $P$, as an indecomposable right $A$-module, is not a non-projective indecomposable G-projective module.
\end{lemma}

\begin{proof}
Assume $P_A \in \npGproj(A)$, then, by Theorem \ref{thm:Kalck}, there exists a full-relational oriented cycle $\w{\C} = a_1'\cdots a_m'$ on the quiver $\Q$ of $A=\kk\Q/\I$ such that $P_A \cong a_r'A$ holds for some $1\le r\le m$. Then we have
\begin{align}\label{formula: iso 1 in lemm:proj}
  P_{\ol{A}_{\C,t}}
\cong a_r' A \otimes_A \ol{A}_{\C,t}
\cong a_r'\ol{A}_{\C,t}.
\end{align}
It follows that
\begin{align}\label{formula: iso 2 in lemm:proj}
  P_{\ol{A}_{\C,t}} \cong \ol{\e}_{\s(a_r')}\ol{A}_{\C,t} \ (=\ol{\e}_v\ol{A}_{\C,t})
\end{align}
since $P_{\ol{A}_{\C,t}}$ is an indecomposable projective right $\ol{A}_{\C,t}$-module.
The isomorphism (\ref{formula: iso 2 in lemm:proj}) yields that $\s(a_r')=v$ is a vertex on the full-relational oriented cycle $\w{\C}$, it contradicts with (a).
Therefore, if $v$ is not a vertex on any full-relational oriented cycle,
then $P_A$ is not a non-projective indecomposable G-projective module.

Next, we structure two contradictions under the conditions (b) and (c), respectively.
By (\ref{formula: iso 1 in lemm:proj}) and (\ref{formula: iso 2 in lemm:proj}) we have
\begin{align}\label{formula: iso 3 in lemm:proj}
  a_r'\ol{A}_{\C,t} \cong \ol{\e}_{\s(a_r')}\ol{A}_{\C,t}
  \ (\cong P_{\ol{A}_{\C,t}}).
\end{align}
We have two cases as follows.

\begin{itemize}
\item[(1)]
If $\w{\C}=\C$, that is, $v$ satisfies (b), then $a_r'$ is an arrow on $\C$ whose starting point is $\s(a_t)$,
i.e., $a_r'=a_t$.
Then
\[0 \ne \ol{\e}_{\s(a_t)}\ol{A}_{\C,t}\ol{\e}_{\s(a_t)}
\cong a_t\ol{A}_{\C,t}\ol{\e}_{\s(a_t)} \]
by (\ref{formula: iso 3 in lemm:proj}).
So there is at least one path $p = \beta_1\cdots\beta_u \notin \I$ from $\t(a_t)$ to $\s(a_t)$ such that $a_tp\e_{\s(a_t)} = a_tp$ is an oriented cycle on $\Q$.
Since $\C$ is a full-relational oriented cycle, we have $a_{t-1}a_t\in\I$.
Notice that $\t(\beta_u)=\t(a_{t-1})=\s(a_t)$, so $\beta_ua_t\notin\I$ by the definition of gentle algebra.
One can check that $a_t\beta_1\notin\I$.
Thus, $a_tp$ is an oriented cycle without relation, this is a contradiction since $A$ is finite-dimensional.

\item[(2)]
If $\w{\C}\ne \C$, that is, $v$ satisfies (c), then we have two subcases as follows.

\begin{itemize}
  \item[(i)] The full-relational oriented cycles $\w{\C}$ and $\C$ do not have any common vertex.
  \item[(ii)] The common vertex of $\w{\C}$ and $\C$ is either $\s(a_t)$ or $\t(a_t)$.
\end{itemize}
In any subcase, $\w{\C}$ is a full-relational oriented cycle on the quiver $\ol{\Q}$ of $\ol{A}_{\C,t}$.
Then, by (\ref{formula: iso 3 in lemm:proj}) and Theorem \ref{thm:Kalck},  we have
\[\ol{\e}_{\s(a_r')}\ol{A}_{\C,t} \cong a_r'\ol{A}_{\C,t}
  \in \ind(\npGproj(\ol{A}_{\C,t})).\]
However, $\ol{\e}_{\s(a_r')}\ol{A}_{\C,t}$ is a projective right $\ol{A}_{\C,t}$-module, a contradiction.
\end{itemize}
\end{proof}


Notice that Lemma \ref{lemm:proj} may be fail if the common vertex of $\C$ and $\w{\C}$ is neither $\s(a_t)$ nor $\t(a_t)$, see the following instance.

\begin{example} \label{examp:counter-examp} \rm
Let $A=\kk\Q/\I$ be the gentle algebra given by the bound quiver $(\Q,\I)$, where the quiver $\Q$ is shown in \Pic \ref{fig:counter-examp}(1), and the admissible ideal is $\I=\langle a_1a_2, a_2a_3,$ $a_3a_1, b_1b_2, b_2b_3, b_3b_1\rangle$.
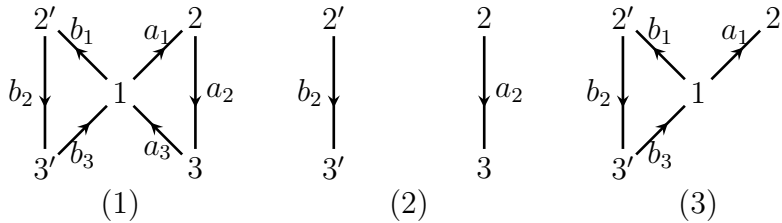
\begin{figure}[htbp]
\centering
\begin{tikzpicture}[scale=1]
\draw[postaction={on each segment={mid arrow=black}}]
  [line width = 1pt]
  ( 0, 0) -- ( 1, 1) -- ( 1,-1) -- ( 0, 0) --
  (-1, 1) -- (-1,-1) -- ( 0, 0);
\fill[white] ( 0, 0) circle(0.25);
\fill[white] ( 1, 1) circle(0.25);
\fill[white] ( 1,-1) circle(0.25);
\fill[white] (-1, 1) circle(0.25);
\fill[white] (-1,-1) circle(0.25);
\draw ( 0, 0) node{$1$};
\draw ( 1, 1) node{$2$};
\draw ( 1,-1) node{$3$};
\draw (-1, 1) node{$2'$};
\draw (-1,-1) node{$3'$};
\draw ( 0.5, 0.5) node[above]{$a_1$};
\draw ( 1.0, 0.0) node[right]{$a_2$};
\draw ( 0.5,-0.5) node[below]{$a_3$};
\draw (-0.5, 0.5) node[above]{$b_1$};
\draw (-1.0, 0.0) node[ left]{$b_2$};
\draw (-0.5,-0.5) node[below]{$b_3$};
\draw (0,-1.5) node{(1)};
\end{tikzpicture}
\ \ \
\begin{tikzpicture}[scale=1]
\draw[postaction={on each segment={mid arrow=black}}]
  [line width = 1pt] ( 1, 1) -- ( 1,-1);
\draw[postaction={on each segment={mid arrow=black}}]
  [line width = 1pt] (-1, 1) -- (-1,-1);
\fill[white] ( 0, 0) circle(0.25);
\fill[white] ( 1, 1) circle(0.25);
\fill[white] ( 1,-1) circle(0.25);
\fill[white] (-1, 1) circle(0.25);
\fill[white] (-1,-1) circle(0.25);
\draw ( 1, 1) node{$2$};
\draw ( 1,-1) node{$3$};
\draw (-1, 1) node{$2'$};
\draw (-1,-1) node{$3'$};
\draw ( 1.0, 0.0) node[right]{$a_2$};
\draw (-1.0, 0.0) node[ left]{$b_2$};
\draw (0,-1.5) node{(2)};
\end{tikzpicture}
\ \ \
\begin{tikzpicture}[scale=1]
\draw[postaction={on each segment={mid arrow=black}}]
  [line width = 1pt]
  ( 0, 0) -- ( 1, 1);
\draw[postaction={on each segment={mid arrow=black}}]
  [line width = 1pt]
  ( 0, 0) -- (-1, 1) -- (-1,-1) -- ( 0, 0);
\fill[white] ( 0, 0) circle(0.25);
\fill[white] ( 1, 1) circle(0.25);
\fill[white] ( 1,-1) circle(0.25);
\fill[white] (-1, 1) circle(0.25);
\fill[white] (-1,-1) circle(0.25);
\draw ( 0, 0) node{$1$};
\draw ( 1, 1) node{$2$};
\draw (-1, 1) node{$2'$};
\draw (-1,-1) node{$3'$};
\draw ( 0.5, 0.5) node[above]{$a_1$};
\draw (-0.5, 0.5) node[above]{$b_1$};
\draw (-1.0, 0.0) node[ left]{$b_2$};
\draw (-0.5,-0.5) node[below]{$b_3$};
\draw (0,-1.5) node{(3)};
\end{tikzpicture}
\caption{The gentle algebra $A$ given in Example \ref{examp:counter-examp} and its quotients $\ol{A}_{\C,2}$ and $\ol{A}_{\C,1}$.}
\label{fig:counter-examp}
\end{figure}
The gentle algebra $A$ has 6 non-projective indecomposable G-projective right $A$-modules (up to isomorphism) as follows.
\begin{itemize}
  \item[(a)] The simple modules $S(2)$, $S(3)$, $S(2')$, $S(3')$ corresponded by the vertices $2,3,2',3'$, respectively.
  \item[(b)] The indecomposable modules $({}_{2}^{1})$ and $({}_{2'}^{1})$.
\end{itemize}

Take $\C=a_1a_2a_3$ and $t=2$. Then $\ol{A}_{\C,2} = A/A\epsilon_{\C,2}A = A/A\e_1A$,
and the quiver of $\ol{A}_{\C,2}$ is shown in \Pic \ref{fig:counter-examp} (2).
Two full-relational oriented cycles $\C=a_1a_2a_3$ and $\w{\C}=b_1b_2b_3$ have a common vertex $1$.
The vertex $3'$ does not satisfy Lemma \ref{lemm:proj} (a), (b), or (c),
and $\ol{\e}_{3'}\ol{A}_{\C,2} = S(3')_{\ol{A}_{\C,2}}$ is an indecomposable projective right $\ol{A}_{\C,2}$-module.
However, $\ol{\e}_{3'}\ol{A}_{\C,2}$, as a right $A$-module, is isomorphic to $S(3')_A$ lying in $\npGproj(A)$.

Next, for the full-relational oriented cycle $\C$, we consider the case of $t=1$.
Then $\ol{A}_{\C,1}=A/A\epsilon_{\C,1}A=A/A\e_3A$,
its quiver is shown in \Pic \ref{fig:gent quotient}(3).
The vertexes $1,2',3'$ satisfies the condition given in Lemma \ref{lemm:proj} (a), and we have that
\begin{center}
  $P(1)_{\ol{A}_{\C,1}} = ({}_{2'}{}^{1}{}_{2})_{\ol{A}_{\C,1}}$,
  $P(2')_{\ol{A}_{\C,1}} = ({}_{3'}^{2'})_{\ol{A}_{\C,1}}$
  and $P(3')_{\ol{A}_{\C,1}} =
  \left(\begin{smallmatrix} 3' \\ 1 \\ 2 \end{smallmatrix} \right)_{\ol{A}_{\C,1}}$
\end{center}
are not non-projective indecomposable G-projective right $A$-modules.
\end{example}

Next, we provide an example for Lemma \ref{lemm:proj}.

\begin{example} \rm
We consider the gentle algebra $A$ given in Example \ref{exp:gent} and its quotient $\ol{A}_{\C,1}$ given in Example \ref{examp:gent-quot}.
Then
\[P(2)_{\ol{A}_{\C,1}} = \ol{\e}_2\ol{A}_{\C,1} =
\left(\begin{smallmatrix}
2 \\ 6 \\ 7
\end{smallmatrix}\right)_{\ol{A}_{\C,1}}
=\left(\begin{smallmatrix}
2 \\ 6 \\ 7
\end{smallmatrix}\right)_{A}
\in \ind(\npGproj(A)).\]
%
%
The vertex $2$, as a vertex in the quiver $\ol{\Q}$ of $\ol{A}_{\C,1}$,
is the ending point of $a_1$, and $P(2)_{\ol{A}_{\C,1}}$, as a right $A$-module, lies in $\ind(\npGproj(A))$.
For other indecomposable projective right $\ol{A}_{\C,1}$ modules, we obtain:
\begin{itemize}
\item[(a)]
the following indecomposable projective right $\ol{A}_{\C,1}$-modules
\begin{align*}
& P(1)_{\ol{A}_{\C,1}} = \left({\begin{smallmatrix}
    & 1 &   \\
  2 &   & 4 \\
  6 &   & 5 \\
  7 &   &
  \end{smallmatrix}}\right)_{\ol{A}_{\C,1}},
&& P(4)_{\ol{A}_{\C,1}} = \big({}^4_5\big)_{\ol{A}_{\C,1}},
&& P(8)_{\ol{A}_{\C,1}} = \big({}^8_9\big)_{\ol{A}_{\C,1}},
\\
&  P(9)_{\ol{A}_{\C,1}} = \big({}^9_2\big)_{\ol{A}_{\C,1}},
&& P(6)_{\ol{A}_{\C,1}} = \big({}^6_7\big)_{\ol{A}_{\C,1}},
&& P(7)_{\ol{A}_{\C,1}} = \big({}^7_1\big)_{\ol{A}_{\C,1}}
\end{align*}
are indecomposable projective right $A$-modules,

\item[(b)]
and the indecomposable projective $\ol{A}_{\C,1}$-module $P(5)_{\ol{A}_{\C,1}} \cong S(5)_{\ol{A}_{\C,1}}$ is not an indecomposable projective $A$-module.
\end{itemize}
All modules given in (a) and (b) do not be non-projective indecomposable G-projective $A$-modules.
\end{example}

\subsection{The subalgebra {\texorpdfstring{$\w{A}_{\C,t}$}.}}
Define $\w{A}_{\C,t} := \epsilon_{\C,t}A\epsilon_{\C,t}$. Then it is a subalgebra of $A$ whose identity is $\epsilon_{\C,t}$.

\begin{lemma} \label{lemm:cycle}
For a full-relational oriented cycle $\C=a_1a_2\cdots a_l$ of the gentle algebra $A=\kk\Q/\I$, the following two statements hold.
\begin{itemize}
  \item[{\rm(1)}] If $a_tA\epsilon_{\C,t}\ne 0$, then the number of all cycles of $\Q$ is greater than or equal to two.
  \item[{\rm(2)}] If $\dim_{\kk}a_tA\e_{t+2}>1$ {\rm(}we take $t+1=1$ in the case of $t=l${\rm)}, then $A$ is representation-infinite.
\end{itemize}
\end{lemma}

\begin{proof}
Assume $\s(a_i)=i$ ($1\le i\le l$).

(1) If $a_tA\epsilon_{\C,t}\ne 0$, then there is at least one vertex $1\le u\le l$, $u\ne \s(a_t)=t$ and $u\ne \t(a_t)=t+1$ such that $a_tA\e_u\ne 0$.
Thus, we can find a path $p=b_1\cdots b_m$ from $\t(a_t)=\s(a_{t+1})=t+1$ to $\t(b_m)=u=\s(a_u)$ which is non-zero on the bound quiver $(\Q,\I)$ of $A$.
Then $q=a_{t+1}\cdots a_{u-1}$, a path on $\Q$, and $p$ form a cycle on $\Q$ by using $\s(q)=\s(a_{t+1})=t+1=\s(p)$ and $\t(q)=\t(a_{u-1})=u=\t(p)$.
It follows that (1) holds.

(2) If $\dim_{\kk}a_tA\e_{t+2}>1$, then, except $a_{t+1}$, there is at least one path $p'$ from $\s(a_{t+1})=t+1$ to $\t(a_{t+1})$ such that $p'\notin\I$ holds.
The paths $p'$ and $a_{t+1}$ form a cycle on $\Q$, and any path on this cycle, as an element in $A$, is non-zero.
This cycle forms a hereditary subquiver of Euclidean type $\w{\mathbb{A}}$.
It follows that (2) holds.
\end{proof}

Now, we provide an instance for Lemma \ref{lemm:cycle}.

\begin{example} \label{examp:cycle} \rm
Let $A=\kk\Q/\I$ be a gentle algebra whose quiver $\Q$ is shown in \Pic \ref{fig:cycle},
and the admissible ideal is $\I = \langle a_1a_2, a_2a_3, a_3a_4, a_4a_1\rangle$.
\begin{figure}[htbp]
\centering
\begin{tikzpicture}
\draw[postaction={on each segment={mid arrow=black}}]
  [line width = 1pt]
  (-1,0.5) -- (1,0.5) -- (1,-1) -- (-1,-1) -- (-1,0.5);
\draw [line width = 1pt]
  (1,-1) arc(-45:225:1.4);
\draw[postaction={on each segment={mid arrow=black}}]
  [line width = 1pt] (0.1,1.38) -- (-0.1,1.38);
\fill[white] ( 1,.5) circle (0.25);
\fill[white] ( 1,-1) circle (0.25);
\fill[white] (-1,-1) circle (0.25);
\fill[white] (-1,.5) circle (0.25);
\draw (1,0.5) node{1} (1,-1) node{2} (-1,-1) node{3} (-1,0.5) node{4};
\draw (1,-.25) node[left]{$a_1$} (0,-1) node[below]{$a_2$}
      (-1,-.25) node[right]{$a_3$} (0,0.5) node[above]{$a_4$}
      (0,1.4) node[above]{$b$};
\end{tikzpicture}
\caption{A gentle algebra with two cycles.}
\label{fig:cycle}
\end{figure}
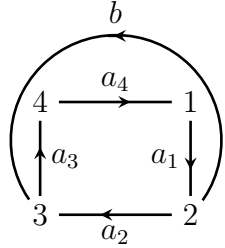
Consider the full-relational oriented cycle $\C=a_1a_2a_3a_4$ and take $t=1$.
Then $\epsilon_{\C,1}=\e_3+\e_4$.

We know $a_1A\e_3 = \kk b \ne 0$ (thus, $a_1A\epsilon_{\C,1}\ne 0$).
It follows that the arrow $b$ (as a path of length one) and the path $a_3a_4a_1$ form a cycle which is not $\C$. As shown in \Pic \ref{fig:cycle}, the quiver of $A$ has two cycles $\C=a_1a_2a_3a_3$ and $ba_3a_4a_1$.

Moreover, for the idempotent $\e_{t+2}=\e_3$, we have $a_1A\e_3 = \kk a_2 + \kk p$ and $\dim_{\kk}a_1A\e_3 = 2$,
that is, the paths $p'=b$ and $a_2$ form a cycle $\mathscr{D}$ which is of the form $\xymatrix{3 & 2 \ar@/^0.5pc/[l]^{a_2} \ar@/_0.5pc/[l]_{b}}$.
$\mathscr{D}$, as a bound subquiver of $(\Q,\I)$, is a 2-Kronecker quiver.
It is well-known that $\kk\mathscr{D}$ is representation-infinite,
then one can check that $A$ is representation-infinite.
\end{example}

\section{Main result} \label{sect:main}
For a gentle algebra $A$ with full-relational oriented cycle $\C=a_1\cdots a_l$, we denote by $\calR_e$ the recollement (\ref{recollement}).
In particular, if $e=\epsilon_{\C,t}$, we denote by $\calR_{\C,t}$ the recollement
\[ \calR_{\epsilon_{\C,t}} :=\ \
\xymatrix@C=2cm{
  \modcat(\ol{A}_{\C,t})
  \ar[r]^{\emb_{\C,t}}_{\text{(embedding)}}
& \modcat(A)
  \ar[r]^{\res_{\C,t} := (-)\epsilon_{\C,t}}_{\text{(retraction)}}
  \ar@/^1.8pc/[l]^{H_{\C,t}:=\Hom_A(\ol{A}_{\C,t},-)}
  \ar@/_1.8pc/[l]_{T_{\C,t}:=-\otimes_A \ol{A}_{\C,t}}
& \modcat(\w{A}_{\C,t}).
  \ar@/_1.8pc/[l]_{\w{T}_{\C,t}:=-\otimes_{\w{A}_{\C,t}} \epsilon_{\C,t} A}
  \ar@/^1.8pc/[l]^{\w{H}_{\C,t}:=\Hom_{\w{A}_{\C,t}}(A\epsilon_{\C,t},-)}
}\]
We define
\[\calR(A) = \{ \calR_e \mid e \text{ is an idempotent of } A \}\]
and
\[\GcalR(A)
:= \bigcup_{
   \begin{smallmatrix}
     \C \text{ is a full-relational} \\
      \text{oriented cycle}
   \end{smallmatrix}
  }
   \{ \calR_{\C, t} \mid 1\le t\le l \}. \]

Now we provide the first main result of our paper,
and we will provide an example for this result, see Example \ref{examp:main 1}.

\begin{theorem} \label{thm:main 1}
Let $A=\kk\Q/\I$ be a gentle algebra with at least one cycle.
If $(\Q,\I)$ has a full-relational oriented cycle,
then there is a injection
\[ \phi: \ind(\npGproj(A)) \to \calR(A) \]
sending each module $G$ lying in $\ind(\npGproj(A))$ to a recollement $\phi(G)$ such that
\begin{itemize}
  \item[\rm(1)] $\phi(G) = \calR_{\C,t}$ for some full-relational oriented cycle $\C=a_1\cdots a_l$ and $1\le t\le l$;
  \item[\rm(2)] $\mathrm{Im}\phi = \GcalR(A)$, and $\varphi: \ind(\npGproj(A)) \to \GcalR(A)$, $G\mapsto\phi(G)$ is a bijection;
  \item[\rm(3)] $T_{\C,t}(G)$ is an indecomposable projective right $\ol{A}_{\C, t}$-module;
  \item[\rm(4)] if there is an $A$-module $G\in\ind(\npGproj(A))$ satisfying $\dim_{\kk} \res_{\C,t}(G)\ge 2$, where $\res_{\C,t}$ is the retraction functor given in the recollement $\phi(G)$, then $A$ is representation-infinite.
\end{itemize}
\end{theorem}

\begin{proof}
(1)+(2) Assume $\ol{A}_{\C, t} = \kk\ol{\Q}/\ol{\I}$ in this proof.
For any $G\in \ind(\npGproj(A))$, we have $G \cong a_tA$ by Theorem \ref{thm:Kalck},
where $\C=a_1\cdots a_l$ is a full-relational oriented cycle and $a_t$ ($1\le t\le l$) is an arrow on $\C$. Define
\[\phi: \ind(\npGproj(A)) \to \calR(A), \ a_tA \mapsto \calR_{\C,t},\]
it is clear that $\phi$ is injective.

On the other hand, for arbitrary full-relational oriented cycle $\C'=a_1'\cdots a_{l'}'$ and any arrow $a_{t'}'$ ($1\le t'\le l'$),
$a_{t'}'A$ is a preimage of the recollement $\calR_{\C',t'}$ under the map $\phi$.
It follows that $\mathrm{Im}\phi$ and $\GcalR(A)$ coincide and
\[\varphi: \ind(\npGproj(A)) \to \GcalR(A)\]
is surjective. Then we construct a map $\phi$ from $\ind(\npGproj(A))$ to $\calR(A)$ such that (1) and (2) hold.

(3) The tensor product $G\otimes_{A}\ol{A}_{\C,t} \cong a_tA\otimes_{A}\ol{A}_{\C,t}$
is isomorphic to the indecomposable projective right $\ol{A}_{\C,t}$-module
$\ol{\e}_{\t(a_t)}\ol{A}_{\C,t}$ by Lemma \ref{lemm:idemp}, that is, (3) holds.

(4)
Lemma \ref{lemm:cycle} provides the statement (4).
\end{proof}

By \cite[Remark 2.5]{Psa2014}, we know that the functor $q$ in the recollement (\ref{recollement diagram}) preserves projective objects.
In general, $q$ may not preserve G-projective objects.
See for example, in \cite[Example 2.5]{ZZ2020}, the authors provided an example to show this fact by using the algebra $A=\kk\Q/\I$ with the quiver
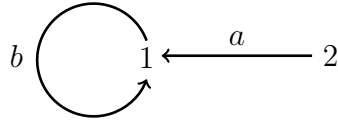
\[\begin{tikzpicture}
\draw[->][line width=1pt] (0,-0.1) arc(20:340:0.75) node[above]{$1$};
\draw (-1.5,-0.3) node[left]{$b$};
\draw[<-][line width=1pt] (0.2,-0.3) -- (2.2,-0.3);
\draw (2.2,-0.3) node[right]{$2$};
\draw (1.2,-0.3) node[above]{$a$};
\end{tikzpicture}\]
and the admissible ideal $\I=\langle ab, a^2 \rangle$.

As an application of Theorem \ref{thm:main 1}, we obtain the following result, which shows that the functor $T_{\C,t}$ in the recollements of the form $\calR_{\C,t}$ defined over some special gentle algebras preserves G-projective objects.


\begin{corollary} \label{coro:of main 1}
Let $A=\kk\Q/\I$ be a gentle algebra.
If arbitrary two full-relational oriented cycles of $A$ have no common vertex, then for any full-relational oriented cycle $\C=a_1\cdots a_l$ and arbitrary $1\le t\le l$, the functor $T_{\C,t}: \modcat(A) \to \modcat(\ol{A}_{\C,t})$ of the recollement $\calR_{\C,t}$ preserves G-projectives.
\end{corollary}


\begin{proof}
Let $G$ be an indecomposable G-projective right $A$-module.
If $G$ is projective, then $T_{\C,t}(G)$ is projective since $T_{\C,t} = - \otimes_A \ol{A}_{\C,t}$ preserves projective objects.
Now, we assume $G\in\ind(\npGproj(A))$, then there is a full-relational oriented cycle $\mathscr{D} = b_1\cdots b_m$ such that $G$ is isomorphic to $b_iA$.
We have two cases as follows:

(1) $\mathscr{D}=\C$;

(2) $\mathscr{D}\ne\C$.

In the case (1), $b_i=a_j$ for some $1\le j\le l$. If $j\ne t$, then the image of the arrow $a_j$ is zero up to the canonical epimorphism $\pi: A\to \ol{A}_{\C,t} = A/A\epsilon_{\C,t}A$.
It follows that $T_{\C,t}(b_iA) = a_{j}A\otimes_A\ol{A}_{\C,t} = \pi(a_j)\ol{A}_{\C,t} = 0$ is projective.
If $j=t$, then $T_{\C,t}(b_iA)=T_{\C,t}(a_tA)$ is projective by Theorem \ref{thm:main 1}(3).

In the case (2), since $\mathscr{D}$ and $\C$ have no common vertex, we obtain that $\pi(\mathscr{D})$, the image of $\mathscr{D}$ under the canonical epimorphism $\pi$, is also a full-relational oriented cycle on the bound quiver of $\ol{A}_{\C,t}$.
Thus, any arrow $b_i$ on $\mathscr{D}$ can be seen as an arrow $\pi(b_i)$ on $\pi(\mathscr{D})$.
Then $T_{\C,t}(b_iA) = b_iA \otimes_A \ol{A}_{\C,t} \cong b_i\ol{A}_{\C,t} \cong \pi(b_i)\ol{A}_{\C,t}$ is a G-projective right $\ol{A}_{\C,t}$-module in $\ind(\npGproj(\ol{A}_{\C,t}))$ by Theorem \ref{thm:Kalck}.
\end{proof}

In Corollary \ref{coro:of main 1}, the condition ``arbitrary two full-relational oriented cycles of $A$ have no common vertex'' is necessary, see Example \ref{examp:coro-of main 1}.

If $A$ is a gentle one-cycle algebra, then
we have some finer precise properties than Theorem \ref{thm:main 1}.

\begin{theorem} \label{thm:main 2}
Assume that $A$ is a gentle one-cycle algebra. If the unique cycle of $(\Q,\I)$, written as $\C=a_1\cdots a_l$
{\rm(}$\s(a_i)=i$, $\forall 1\le i\le l${\rm)}, is full-relational oriented, then the following statements hold.
\begin{itemize}
  \item[\rm(1)] $T_{\C,t}(a_uA) \cong
    \begin{cases}
      \ol{e}_{\t(a_t)}\ol{A}_{\C,t}, & \text{if } u=t; \\
      0, & \text{if } u\ne t
    \end{cases}$ holds for all $1\le u\le l$;
  \item[\rm(2)] for any indecomposable projective right $\ol{A}_{\C,t}$-module $P$
    satisfying $P \not\cong \ol{e}_{\t(a_t)}\ol{A}_{\C,t}$,
    we have $\emb_{\C,t}(P)\notin \ind(\npGproj(A))$;
  \item[\rm(3)] $\emb_{\C,t}(T_{\C,t}(a_tA)) \in \ind(\npGproj(A))$. 
\end{itemize}
\end{theorem}

\begin{proof}
Let $P = \ol{\e}_v\ol{A}_{\C,t}$ be an indecomposable projective right $\ol{A}_{\C,t}$-module which does not be isomorphic to $\ol{\e}_{\t(a_t)}\ol{A}_{\C,t}$.
The case of $u=t$ in the statement (1) is a direct corollary of Theorem \ref{thm:main 1} (3),
and the case of $u\ne t$ in the statement (1) holds since any arrow $a_u$ as an element in $\ol{A}_{\C,t}$ equals to zero.
Next, we show (2) and (3).

(2) If $v\ne\s(a_t)$, then $v$ must be a vertex which is not on $\C$.
Since $A$ is gentle one-cycle, the statement (1.2) holds by Lemma \ref{lemm:proj}.
If $v=\s(a_t)$ ($=\t(a_{t-1})$, here, $t-1=l$ if $t=1$, we assume $P=\ol{\e}_v\ol{A}_{\C,t} \in \ind(\npGproj(A))$,
then, by Theorem \ref{thm:Kalck} and $\dim_{\kk}P\ol{\e}_{\t(a_{t-1})} = \dim_{\kk}P\ol{\e}_{v}$ $(=\dim_{\kk}P\e_{v})$ $\ne 0$,
we have $P\cong a_{t-1}A$. immediately, we obtain
\[ \dim_{\kk} \ol{\e}_{\s(a_t)}\ol{A}_{\C,t}\e_{\s(a_{t-1})}
= \dim_{\kk} P\e_{\s(a_{t-1})}
= \dim_{\kk} a_{t-1}A\e_{\s(a_{t-1})} \ne 0. \]
It follows that the length $l$ of $\C$ equals to two, which contradicts with Assumption \ref{assump}.
Thus, the statement (2) holds.

(3) We have
\begin{align} \label{formula: iso in thm 1}
 T_{\C,t}(a_tA) = a_t A\otimes_{A} \ol{A}_{\C,t} \cong \ol{\e}_{\t(a_t)}\ol{A}_{\C,t}
 = \bigoplus_{
     \begin{smallmatrix}
       {\s(p)=\t(a_t)} \\
       {p \text{ is a non-zero path on } (\ol{\Q},\ol{\I})}
     \end{smallmatrix}
 } \kk p
\end{align}
and
\begin{align} \label{formula: iso in thm 2}
a_t A = \bigoplus_{
  \begin{smallmatrix}
    {\s(q)=\t(a_t)} \\ {q \text{ is a non-zero path on } (\Q,\I)}
  \end{smallmatrix}
  } \kk a_tq
\end{align}
by the definition of finite-dimensional algebra.
Let $X_1$ be the set of all direct summands of (\ref{formula: iso in thm 1})
and $X_2$ be that of all direct summands of (\ref{formula: iso in thm 2}). Then the map
\[ f: X_1 \to X_2, \ \kk p\mapsto \kk a_tp\]
is injective by using the following two facts:
\begin{itemize}
  \item Any path on the quiver $\ol{\Q}$ of $\ol{A}_{\C,t}$ can be seen as a path on the quiver $\Q$ of $A$ (or equivalently, each right $\ol{A}_{\C,t}$-module naturally is a right $A$-module);
  \item $a_tp\ne 0$ holds for any $p\in X_1$. Otherwise,
    assume $p=b_1\cdots b_{\ell}$ ($b_i\in\Q_1$ for all $1\le i\le \ell$, $b_1\ne a_{t+1}$),
    we obtain $a_tb_1\in\I$. It contradict with $A$ to be a gentle algebra.
\end{itemize}
Moreover, for each $a_tq \in X_2$, the path $q$ does not cross any vertex lying in $\{1,\cdots, t-1,t+2, \cdots,l\}$,
then $q$ is also a path on $(\ol{\Q}.\ol{\I})$ since $\Q$ has only one cycle $\C$,
i.e., $q\in X_1$. We obtain $f(q)=a_1q$. Thus, $f$ is surjective. It follows that $f$ is a bijection.
Then the map $f$ induces a homomorphism between $T_{\C,t}(a_tA)$ and $a_tA$ which naturally is an isomorphism as required.
\end{proof}

Now we provide a remark for Theorem \ref{thm:main 2}.

\begin{remark} \label{rmk: rmk of main 2} \rm
(1) For a gentle one-cycle algebra $A$, it is clear that $\ind(\npGproj(A)) \ne \varnothing$ admits that $A$ is representation-finite. 
Indeed, without loss of generality, assume $b_tA\in\ind(\npGproj(A))$ by Theorem \ref{thm:Kalck},
where $b_t$ is an arrow on some full-relational oriented cycle $\mathscr{D} = b_1\cdots b_{\ell}$.
Since $A$ is gentle one-cycle, $\mathscr{D}$ is the unique cycle on $\Q$.
It is well-known that $A$ is representation-infinite if and only if
$\mathscr{D}$, as a bound subquiver of $(\Q,\I)$, is a Euclidean type of $\w{\mathbb{A}}$ without relation
(or equivalently, if and only if $(\Q,\I)$ contains a band, see \cite[Theorem in page 161]{BR1987}),
we obtain a contradiction.

(2) For any gentle one-cycle algebra $A=\kk\Q/\I$ with full-relational oriented cycle $\C = a_1 \cdots a_l$,
by Theorem \ref{thm:main 2} (1) and the fact that
$T_{\C,t}$ preserves projectives, we have that
\[ T_{\C,t}|_{\Gproj(A)}: \Gproj(A) \to \proj(\ol{A}_{\C,t}),
\ G \mapsto G\otimes_A \ol{A}_{\C,t}\]
is a surjection.
Since it is trivial that $\proj(\ol{A}_{\C,t}) \subseteq \Gproj(\ol{A}_{\C,t})$ (or precisely, $\proj(\ol{A}_{\C,t}) = \Gproj(\ol{A}_{\C,t})$ in the case of $A$ to be gentle one-cycle),
we obtain that $T_{\C,t}|_{\Gproj(A)}$ sends each indecomposable G-projective right $A$-module to an indecomposable G-projective right $\ol{A}_{\C,t}$-module.
In Example \ref{examp:rmk of main 2}, we show that it is necessary that $A$ is gentle one-cycle.

(3) For two Artinian algebras $\alg_1$ and $\alg_2$, let $F: \modcat(\alg_1) \to \modcat(\alg_2)$ be a functor preserving projective objects and admitting a right adjoint functor $G$.
Lu gave some sufficient conditions (i.e., the conditions $\spadesuit$ and $\clubsuit$) in \cite[Lemma 3.11]{Lu2017} such that $F$ preserves G-projective objects.
The conditions in Corollary \ref{coro:of main 1} and Theorem
\ref{thm:main 2} are different from those of Lu.
For example, let $A=\kk\Q/\I$ be a gentle algebra given by the quiver $Q$
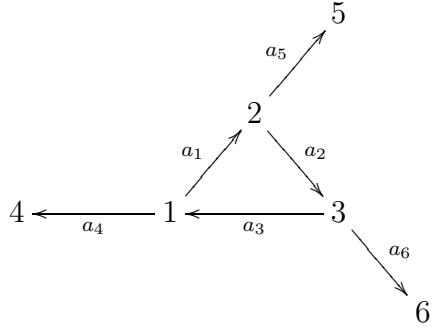
$$\xymatrix@C=20pt{&&&&5\\
&&&2\ar[ru]^{a_5}\ar[rd]^{a_2}\\
4&&\ar[ll]^{a_4}1\ar[ru]^{a_1}&&\ar[ll]^{a_3}3\ar[rd]^{a_6}\\
&&&&&6}$$
and the admissible ideal $\I=\langle a_1a_2, a_2a_3, a_3a_1\rangle$.
Take the full-relational oriented cycle $\C=a_1a_2a_3$ and $t=3$. In the recollement $\calR_{\C,3}$, $T_{\C,3}: \modcat(A) \to \modcat(\ol{A}_{\C,3})$ preserves projectives and adimits a right adjoint functor.
Next, we show that $T_{\C,3}$ does not satisfy the conditions $\spadesuit$ and $\clubsuit$.
\begin{itemize}
  \item For the indecomposable projective right $\ol{A}_{\C,3}$-modules
    $P(1)_{\ol{A}_{\C,3}} = (^{1}_{4})_{\ol{A}_{\C,3}}$ and
    $P(3)_{\ol{A}_{\C,3}} = (^{3}_{6})_{\ol{A}_{\C,3}} \cong a_2A$,
    we have the following short exact sequence
    \[0 \longrightarrow (^{1}_{4})_A \cong T_{\C,3}((^{1}_{4})_{\ol{A}_{\C,3}})
        \longrightarrow P(3)_A
        \longrightarrow a_2A \cong T_{\C,3}((^{3}_{6})_{\ol{A}_{\C,3}})
        \longrightarrow 0 \]
    in $\modcat A$. It follows that $\Ext_A^1(a_2A, (^{1}_{4})_A) \ne 0$,
    i.e., $\spadesuit$ not holds.
  \item One can check $\mathrm{proj.dim}(T_{\C,3}((^{1}_{4})_{\ol{A}_{\C,3}})) = \infty$
  and $\mathrm{inj.dim}(T_{\C,3}((^{1}_{4})_{\ol{A}_{\C,3}})) = \infty$,
  then $\clubsuit$ not holds.
\end{itemize}
However, $T_{\C,t}$ preserves G-projectives by Theorem \ref{thm:main 2}.
\end{remark}

\section{Example} \label{sect:examp}

Finally, we will provide some examples to expalain the obtained results in Section \ref{sect:main}.

\begin{example} \rm \label{examp:main 1}
Consider the gentle algebra $A=\kk\Q/\I$ given in Example \ref{exp:gent}, where $\Q$ is shown in \Pic \ref{fig:gent} and $\I=\langle a_1a_2, a_2a_3, a_3a_1, b_1c_1, b_2c_2, b_3c_3, d_1b_3, d_2b_1, d_3b_2\rangle$.
We have three non-projective indecomposable G-projective right $A$-modules:
\begin{center}
$a_1A \cong \left(
\begin{smallmatrix}
2\\ 6\\ 7
\end{smallmatrix}
\right)_A$,
$a_2A \cong \left(
\begin{smallmatrix}
3\\ 8\\ 9
\end{smallmatrix}
\right)_A$,
and
$a_3A \cong \left(
\begin{smallmatrix}
1\\ 4\\ 5
\end{smallmatrix}
\right)_A$,
\end{center}
by Theorem \ref{thm:main 1},
they correspond to three recollements $\calR_{\C,1}$, $\calR_{\C,2}$, and $\calR_{\C,3}$, respectively.
Here, $\C = a_1a_2a_3$. See \Pic \ref{fig:mian examp},
\begin{figure}[htbp]
\centering
\small
\begin{tikzpicture}[scale = 0.45]
\draw[blue][line width=2pt][dotted] ( 1.73,-1.  ) circle(0.65);
\draw[blue] (1.73+0.65,-1) node[right]{$\w{A}_{\C,t}=\e_3A\e_3$};
\draw[red][line width=24pt]
  ( 0.  , 2.  ) -- (-1.73,-1.  ) -- (-7.82,-1.  ) to[out=45,in=180] (0, 2) --
  ( 3.23, 7.62);
\draw[red][line width=24pt]
  (-1.73,-1.  ) to[out=-60,in=165] (4.71,-6.31) -- (2.71,-3.03);
\draw[white][line width=21pt]
  ( 0.  , 2.  ) -- (-1.73,-1.  ) -- (-7.82,-1.  ) to[out=45,in=180] (0, 2) --
  ( 3.18, 7.52);
\draw[white][line width=21pt]
  (-1.73,-1.  ) to[out=-60,in=165] (4.71,-6.31) -- (2.81,-3.18);
\draw[red][opacity=0.25][line width=12pt]
  (-8.07,-1.  ) -- (-1.73,-1.  ) -- ( 1.73,-1.  ) -- (5.01,-6.84);
\draw[red!50][->]
  (5.01,-6.84) -- (7.01,-6.84) node[right]{\color{red!50}$P(2)_A$};
\draw[orange][line width=1pt]
  (-4.87,-1.00) ellipse (3.7 and 1.3);
\draw[orange][->] (-4.87,-2.30) -- (-4.87,-3.50) node[below]{$a_1A$};
\draw[orange] (-5.87,-5.50) node{$a_1A = P(2)_{\ol{A}_{\C,1}} \cong T_{\C,1}(P(2)_A)$};
\draw[rotate around={60:(0,2)}]
  [line width=1pt]
  [postaction={on each segment={mid arrow=black}}]
  (0,2) -- (2,2);
\draw[rotate around={60:(0,2)}]
  (2.5,2) node{$4$};
\draw[rotate around={60:(0,2)}]
  [line width=1pt]
  [postaction={on each segment={mid arrow=black}}]
  (3,2) -- (5,2);
\draw[rotate around={60:(0,2)}]
  (5.5,2) node{$5$};
\draw[line width=1pt]
  [postaction={on each segment={mid arrow=black}}]
  (-2.23,-1)--(-4.23,-1);
\draw (-4.23,-1) node[left]{$6$};
\draw[line width=1pt]
  [postaction={on each segment={mid arrow=black}}]
  (-5.23,-1)--(-7.23,-1);
\draw (-7.23,-1) node[left]{$7$};
\draw[rotate=120]
  [line width=1pt]
  [postaction={on each segment={mid arrow=black}}]
  (-2.23,-1)--(-4.23,-1);
\draw[rotate=120] (-4.73,-1) node{$8$};
\draw[rotate=120]
  [line width=1pt]
  [postaction={on each segment={mid arrow=black}}]
  (-5.23,-1)--(-7.23,-1);
\draw[rotate=120] (-7.73,-1) node{$9$};
\draw[rotate =  0][line width=1pt]
  [postaction={on each segment={mid arrow=black}}]
   (-7.73,-1) to[out=45,in=180] (0, 2);
\draw[rotate =120][line width=1pt]
  [postaction={on each segment={mid arrow=black}}]
   (-7.73,-1) to[out=45,in=180] (0, 2);
\draw[rotate =240][line width=1pt]
  [postaction={on each segment={mid arrow=black}}]
   (-7.73,-1) to[out=45,in=180] (0, 2);
\draw[rotate =  0](-3.5,2  ) node{$b_1$};
\draw[rotate =120](-3.5,2  ) node{$b_2$};
\draw[rotate =240](-3.5,2  ) node{$b_3$};
\draw[rotate =  0]( 0.2,3.2) node{$c_1$};
\draw[rotate =120]( 0.2,3.2) node{$c_2$};
\draw[rotate =240]( 0.2,3.2) node{$c_3$};
\draw[rotate =  0]( 1.4,5.8) node{$d_1$};
\draw[rotate =120]( 1.4,5.8) node{$d_2$};
\draw[rotate =240]( 1.4,5.8) node{$d_3$};
\draw[rotate=240][postaction={on each segment={mid arrow=black}}]
     [line width=1pt]
     (-1.73+0.5,-1.  ) -- ( 1.73-0.5,-1.  );
\draw(-0.86, 0.5 ) node[left]{$a_1$};
\draw[rotate=  0][postaction={on each segment={mid arrow=black}}]
     [line width=1pt]
     (-1.73+0.5,-1.  ) -- ( 1.73-0.5,-1.  );
\draw[rotate=120](-0.86, 0.5 ) node[below]{$a_2$};
\draw[rotate=120][postaction={on each segment={mid arrow=black}}]
     [line width=1pt]
     (-1.73+0.5,-1.  ) -- ( 1.73-0.5,-1.  );
\draw[rotate=240](-0.86, 0.5 ) node[right]{$a_3$};
\fill[rotate=  0][white]( 0.  , 2.  ) circle (0.35);
\draw[rotate=  0]( 0.  , 2.  ) node{$1$};
\draw[rotate=120]( 0.  , 2.  ) node{$2$};
\draw[rotate=240]( 0.  , 2.  ) node{$3$};
\draw[red] (-8.95,-1) node[left]{$\ol{A}_{\C,1}=A/A\e_3A$};
\end{tikzpicture}
\caption{The recollement $\calR_{\C,1}$ corresponded by $a_1A$.}
\label{fig:mian examp}
\end{figure}
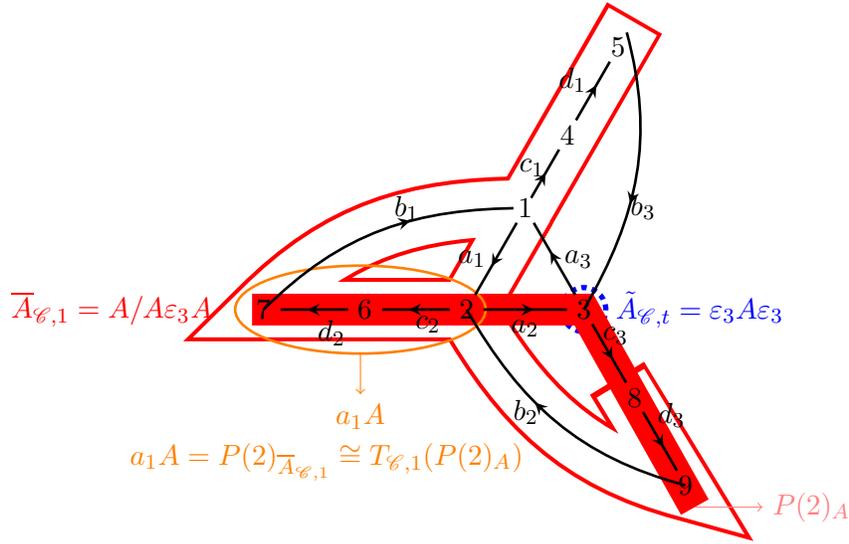
the $c_2d_2$ is the path corresponding to $a_1A$,
it can be seen as a path in the quiver of $\ol{A}_{\C,1}$,
if this case, $c_2d_2$ corresponds to the indecomposable projective right $\ol{A}_{\C,1}$ module $P(2)_{\ol{A}_{\C,1}}$, and we have
\[ T_{\C,1}(a_1A) \cong T_{\C,1}(P(2)_A) = P(2)_{\ol{A}_{\C,1}}. \]
Moreover, $A$ is representation-finite, and one can check that
\begin{center}
$\dim_{\kk} \res_{\C,1}(a_1A) = 0$,
$\dim_{\kk} \res_{\C,1}(a_2A) = \dim_{\kk}S(3) = 1$,
and $\dim_{\kk} \res_{\C,1}(a_3A) = 0$
\end{center}
are less than or equal to $1$.
\end{example}

\begin{example} \label{examp:coro-of main 1} \rm
Let $A=\kk\Q/\I$ be a gentle algebra whose quiver is given by \Pic \ref{fig:coro-of main 1} and the admissible ideal $\I=\langle a_1a_2, a_2a_3, a_3a_1, b_1b_2, b_2b_3, b_3b_4, b_4b_5, b_5b_1 \rangle$.
Then $(\Q,\I)$ has two oriented cycle $\C=a_1a_2a_3$ and $\mathscr{D} = b_1b_2b_3b_4b_5$ which are full-relational.
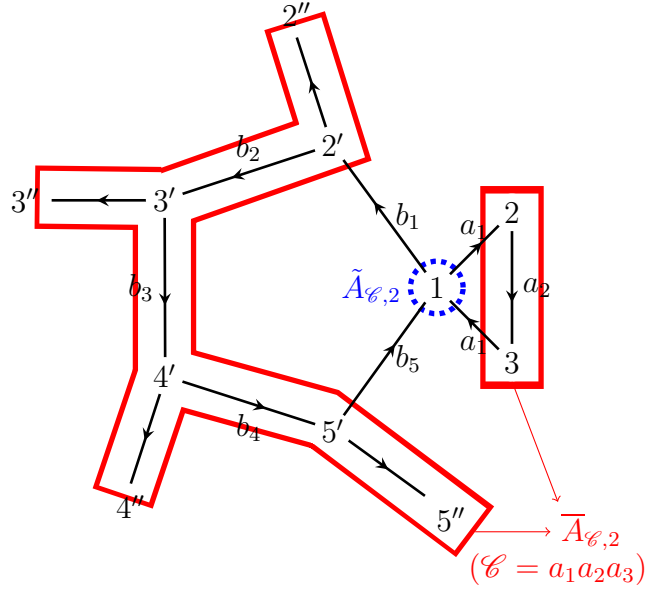
\begin{figure}[htbp]
\centering
\begin{tikzpicture}
\draw[red][rotate =  0][line width=24pt]
   ( 0.11, 3.55) -- ( 0.61, 1.95) -- (-1.62, 1.16) --
   (-1.62,-1.16) -- ( 0.51,-1.75) -- ( 2.51,-3.25);
\draw[red][rotate = 72][line width=24pt]
   ( 0.11, 3.55) -- ( 0.61, 1.95);
\draw[red][rotate =144][line width=24pt]
   ( 0.11, 3.55) -- ( 0.61, 1.95);
\draw[white][rotate =   0][line width=20pt]
   ( 0.13, 3.49) -- ( 0.61, 1.95) -- (-1.62, 1.16) --
   (-1.62,-1.16) -- ( 0.51,-1.75) -- ( 2.45,-3.22);
\draw[white][rotate =  72][line width=20pt]
   ( 0.13, 3.49) -- ( 0.61, 1.95);
\draw[white][rotate = 144][line width=20pt]
   ( 0.13, 3.49) -- ( 0.61, 1.95);
\draw[ red ][line width=24pt] (3,1.35) -- (3,-1.35);
\draw[white][line width=20pt] (3,1.25) -- (3,-1.25);
\draw[red][->] ( 2.51,-3.25) -- ( 3.51,-3.25) node[right]{$\ol{A}_{\C,2}$};
\draw[red][->] (3,-1.25) -- ( 3.61,-2.85);
\draw[red] ( 3.61,-3.75) node{($\C=a_1a_2a_3$)};
\draw[blue][dotted][line width=2pt]
  (2,0) circle (0.35cm)
  (1.7,0) node[left]{$\w{A}_{\C,2}$};
\draw[rotate =   0][line width=1pt]
  [postaction={on each segment={mid arrow=black}}] (2,0) -- (0.61, 1.90);
\draw[rotate =  72][line width=1pt]
  [postaction={on each segment={mid arrow=black}}] (2,0) -- (0.61, 1.90);
\draw[rotate = 144][line width=1pt]
  [postaction={on each segment={mid arrow=black}}] (2,0) -- (0.61, 1.90);
\draw[rotate = 216][line width=1pt]
  [postaction={on each segment={mid arrow=black}}] (2,0) -- (0.61, 1.90);
\draw[rotate = 288][line width=1pt]
  [postaction={on each segment={mid arrow=black}}] (2,0) -- (0.61, 1.90);
\draw[line width=1pt]
  [postaction={on each segment={mid arrow=black}}]
  (2,0) -- (3,1) -- (3,-1) -- (2,0);
\draw[line width=1pt][rotate = -72]
  [postaction={on each segment={mid arrow=black}}]
  (-1.62,1.16) -- (-3.12,1.16) node[above]{$2''$};
\draw[line width=1pt][rotate =   0]
  [postaction={on each segment={mid arrow=black}}]
  (-1.62,1.16) -- (-3.12,1.16) node[left]{$3''$};
\draw[line width=1pt][rotate =  72]
  [postaction={on each segment={mid arrow=black}}]
  (-1.62,1.16) -- (-3.12,1.16) node[below]{$4''$};
\draw[line width=1pt][rotate = 144]
  [postaction={on each segment={mid arrow=black}}]
  (-1.62,1.16) -- (-3.12,1.16) node[below right]{$5''$};
\fill[white][rotate =   0] (2,0) circle (0.25);
\fill[white][rotate =  72] (2,0) circle (0.25);
\fill[white][rotate = 144] (2,0) circle (0.25);
\fill[white][rotate = 216] (2,0) circle (0.25);
\fill[white][rotate = 288] (2,0) circle (0.25);
\fill[white] (3, 1) circle (0.25);
\fill[white] (3,-1) circle (0.25);
\draw[rotate =   0] (2,0) node{$1$};
\draw[rotate =  72] (2,0) node{$2'$};
\draw[rotate = 144] (2,0) node{$3'$};
\draw[rotate = 216] (2,0) node{$4'$};
\draw[rotate = 288] (2,0) node{$5'$};
\draw (3, 1) node{$2$} (3,-1) node{$3$};
\draw (3,0) node[right]{$a_2$};
\draw ( 2.5, 0.5) node[above]{$a_1$};
\draw ( 2.5,-0.5) node[below]{$a_1$};
\draw[rotate =   0]
 (0.5*2.61, 0.5*1.90) node[right]{$b_1$};
\draw[rotate =  72]
 (0.5*2.61, 0.5*1.90) node[above]{$b_2$};
\draw[rotate = 144]
 (0.5*2.61, 0.5*1.90) node[left]{$b_3$};
\draw[rotate = 216]
 (0.5*2.61, 0.5*1.90) node[below]{$b_4$};
\draw[rotate = 288]
 (0.5*2.61, 0.5*1.90) node[right]{$b_5$};
\end{tikzpicture}
\caption{The gentle algebra given in Example \ref{examp:coro-of main 1} and its recollement $\calR_{\C,2}$.}
\label{fig:coro-of main 1}
\end{figure}
The vertex $1$ is a common vertex of $\C$ and $\mathscr{D}$.
Now, we consider the recollement $\calR_{\C,2}$, then the algebras $\ol{A}_{\C,2}=A/A\e_1A$ and $\w{A}_{\C,2}=\e_1A\e_1$ are shown in the {solid} part and {dashed} part, respectively.
One can check that the functor $T_{\C,2}$ sends $b_2A$ ($\in\npGproj(A)$) to the indecomposable right $\ol{A}_{\C,2}$-module
$b_2\ol{A}_{\C,2}
= \left(\begin{smallmatrix}
   3' \\ 3''
  \end{smallmatrix}\right)_{\ol{A}_{\C,2}}$
which is not G-projective.
This example  shows that the condition ``arbitrary two full-relational oriented cycles of a gentle algebra have no common vertex'' in Corollary \ref{coro:of main 1} is necessary.
\end{example}

In Remark \ref{rmk: rmk of main 2}, we show that $T_{\C,t}|_{\Gproj(A)}: \Gproj(A) \to \proj(\ol{A}_{\C,t})$ is surjective if $A$ is a gentle one-cycle algebra. The following example shows that if the number of cycles is greater than or equal to $2$, then $T_{\C,t}(G)$ may be not a projective right $\ol{A}_{\C,t}$-module for some $G\in \ind(\npGproj(A))$.

\begin{example} \rm \label{examp:rmk of main 2}
Consider the gentle algebra $A=\kk\Q/\I$ given in Example \ref{examp:counter-examp} (that is, $\Q$ is shown in \Pic \ref{fig:counter-examp} and $\I$ is generated by $a_1a_2$, $a_2a_3$, $a_3a_1$, $b_1b_2$, $b_2b_3$, $b_3b_1$).
Take $\C=b_1b_2b_3$ and $t=1$, then $\ol{A}_{\C,1} = A/A\e_3 A$.

By Theorem \ref{thm:Kalck}, $A$ has six non-projective indecomposable G-projective right $A$-modules:
\begin{align*}
  &  a_1A \cong S(2)_A,
  && a_2A \cong S(3)_A,
  && a_3A \cong \left({}^{1}_{2'}\right)_A,
  \\
  &  b_1A \cong S(2')_A,
  && b_2A \cong S(3')_A,
  && b_3A \cong \left({}^{1}_{2}\right)_A.
\end{align*}
We have $T_{\C,1}(a_1A) = a_1A\otimes_A \ol{A}_{\C,1} \cong P(2)_{\ol{A}_{\C,1}}$, $T_{\C,1}(a_2A)=0$ and $T_{\C,1}(a_3A)=0$ which are projective in $\proj(\ol{A}_{\C,1})$.
However, one can check that
\begin{align*}
 & T_{\C,1}(b_1A) \cong S(2')_{\ol{A}_{\C,1}}
   \not\cong P(2')_{\ol{A}_{\C,1}}
   = \left({_{3'}^{2'}}\right)_{\ol{A}_{\C,1}}, \\
 & T_{\C,1}(b_2A) \cong S(3')_{\ol{A}_{\C,1}}
   \not\cong P(3')_{\ol{A}_{\C,1}}
   = \left(\begin{smallmatrix}
       3' \\ 1 \\ 2
   \end{smallmatrix}\right)_{\ol{A}_{\C,1}}, \\
\text{and }
  & T_{\C,1}(b_3A) \cong \left({}^1_2\right)_{\ol{A}_{\C,1}}
    \not\cong P(1)_{\ol{A}_{\C,1}}
    = \left({{}_{2'}} {{}^{1}} {{}_{2}} \right)_{\ol{A}_{\C,1}}.
\end{align*}
Right $\ol{A}_{\C,1}$-modules $T_{\C,1}(b_1A)$, $T_{\C,1}(b_2A)$, and $T_{\C,1}(b_3A)$ are not in $\proj(\ol{A}_{\C,1})$.
However, they are non-projective indecomposable G-projective right $\ol{A}_{\C,1}$-modules.
Then,
the functor $T_{\C,1}: \modcat(A) \to \modcat(\ol{A}_{\C,1})$ preserves G-projectives.
which satisfies the conclusion of Corollary \ref{coro:of main 1},
but $A$ has two full-relational oriented cycles which have the common vertex $1$.
\end{example}


\section*{Acknowledgements}

The authors would like to thank Martin Kalck and Yongyun Qin for helpful discussions. 

\section*{Data Availability}

Data sharing not applicable to this article as no datasets were generated or analysed during
the current study.

\section*{Funding}

Yu-Zhe Liu is supported by
National Natural Science Foundation of China (Grant Nos. 12401042, 12171207),
Guizhou Provincial Basic Research Program (Natural Science) (Grant No. ZK[2024]YiBan066)
and Scientific Research Foundation of Guizhou University (Grant Nos. [2023]16, [2022]53, [2022]65).
Dajun Liu is supported by the National Natural Science Foundation of
China (No. 12101003), the Natural Science Foundation of Anhui province (No. 2108085QA07).
Xin Ma is supported by Henan University of Engineering (DKJ2019010).


\def\cprime{$'$}






\end{document}